\documentclass[a4paper,12pt]{amsart}
\usepackage[utf8]{inputenc}
\usepackage[top=1.5in, bottom=1in, left=0.9in, right=0.9in]{geometry}
\usepackage{bbm}
\usepackage{amsthm,amsmath,amssymb,cite,enumerate}
\usepackage{color}

\newtheorem{theorem}{Theorem}[section]
\newtheorem{lemma}[theorem]{Lemma}

\newtheorem{corollary}[theorem]{Corollary}
\newtheorem{definition}[theorem]{Definition}
\newtheorem{conjecture}[theorem]{Conjecture}

\theoremstyle{remark}
\newtheorem{example}{Example}
\newtheorem{remark}{Remark}

\newcommand{\mm}{\mathfrak{m}}
\DeclareMathOperator{\Ass}{Ass}
\DeclareMathOperator{\Spec}{Spec}
\DeclareMathOperator{\Ho}{H}
\DeclareMathOperator{\depth}{depth}

\title{Symbolic powers of ideals defining F-pure and strongly F-regular rings}
\author{Elo\'isa Grifo}
\address{Department of Mathematics, University of Virginia, Charlottesville, VA 22904-4135, USA}
\email{eloisa.grifo@virginia.edu}

\author{Craig Huneke}
\address{Department of Mathematics, University of Virginia, Charlottesville, VA 22904-4135, USA}
\email{huneke@virginia.edu}

\begin{document}


\maketitle 

\begin{abstract}
Given a radical ideal $I$ in a regular ring $R$, the Containment Problem of symbolic and ordinary powers of $I$ consists of determining when the containment $I^{(a)} \subseteq I^b$ holds. By work of Ein-Lazersfeld-Smith, Hochster-Huneke and Ma-Schwede, there is a uniform answer to this question, but the resulting containments are not necessarily best possible. We show that a conjecture of Harbourne holds when $R/I$ is F-pure, and prove tighter containments in the case when $R/I$ is strongly F-regular.
\end{abstract}

\maketitle

\section{Introduction}

In the last decade, there has been growing interest in the comparison problem of symbolic and ordinary powers of ideals. This interest has been in large part motivated by the following containment result, first shown by Ein, Lazarsfeld and Smith \cite{ELS} for rings essentially of finite type over the complex numbers, later extended by Hochster and the second author \cite{comparison} to rings containing a field, and recently shown by Ma and Schwede in the mixed characteristic case \cite{MaSchwede}: given an ideal $I$ in a regular ring $R$, if $h$ is the largest height of an associated prime of $I$, then $I^{(hn)} \subseteq I^n$. However, this containment in general may not be best possible, and finding the smallest symbolic power that is contained in a fixed power of $I$ can be a fairly delicate question. The second author asked
if $I^{(3)} \subseteq I^2$ for unmixed height two reduced ideals in a regular ring. A conjecture raised by Harbourne suggests more finely tuned estimates:

\begin{conjecture}[Harbourne, see \cite{HaHu} and 8.4.3 \cite{Seshadri}]\label{Harbourne}
Let $I$ be a radical homogeneous ideal in $k[\mathbb{P}^N]$, and let $h$ be the maximal height of an associated prime of $I$.\footnote{In \cite{Seshadri} this conjecture is even stated for non-reduced ideals. However, in all our results, the ideals will necessarily be reduced.} Then for all $n \geqslant 1$, 
$$I^{(hn-h+1)} \subseteq I^n.$$
\end{conjecture}

This conjecture holds for monomial ideals \cite[Example 8.4.5]{Seshadri}, star configurations \cite[Example 8.4.8]{Seshadri}, and general sets of points in $\mathbb{P}^2$ \cite{BoH} and in $\mathbb{P}^3$ \cite{Dumnicki2015}. However, there are homogeneous radical ideals of codimension $2$ in $k[x,y,z]$ for which the containment $I^{(3)} \subseteq I^2$ does not hold, showing Conjecture \ref{Harbourne} is false in general: a counterexample was first found by Dumnicki, Szemberg and Tutaj-Gasi\'{n}ska in \cite{counterexamples}, and later extended to a larger class of ideals by Harbourne and Seceleanu in \cite{HaSe}. We do not  know of any prime counterexamples to Conjecture \ref{Harbourne}, nor do we
know any examples where Conjecture \ref{Harbourne} does not hold for $n \gg 0$.

In proving the main result of \cite{comparison}, a stronger containment was shown:  for any $k \geqslant 0$, $I^{(hn + kn)} \subseteq (I^{(k+1)})^n$. Takagi and Yoshida \cite{takagiyoshida} improved this result for the case when $R/I$ is F-pure: $I^{(hn + kn-1)} \subseteq (I^{(k+1)})^n$. As a corollary, one obtains that $I^{(hn-1)} \subseteq I^n$, a result that was independently shown in \cite[Theorem 3.6]{HHfine}. In particular, this implies that Conjecture \ref{Harbourne} holds when $h=2$ and $R/I$ is F-pure.

Despite the results of \cite{takagiyoshida} and \cite{HHfine} mentioned above, in general the singularities of $R/I$ have not been
taken into account in trying to prove stronger bounds. We consider both the case in which $R/I$ is
F-pure (in characteristic $p$) or strongly F-regular (in characteristic $p$). By using standard techniques of reduction to characteristic $p$, we obtain similar results when $R$ is a regular ring essentially of finite type
over the complex numbers and $R/I$ is either dense F-pure type or a log terminal singularity.

Our first main result proves that Harbourne's Conjecture holds if $R/I$ is F-pure. 
Explicitly:

\begin{theorem}[Theorem \ref{Fpure}]
Let $R$ be a regular ring of characteristic $p>0$ (respectively essentially of finite type over
a field of characteristic $0$). Let $I$ be an ideal in $R$ with $R/I$ F-pure (respectively, of dense F-pure type), and let $h$ be the maximal height of a minimal prime of $I$. Then for all $n \geqslant 1$, $I^{(hn-h+1)} \subseteq I^n$.
\end{theorem}

Our second main result gives tighter containments if $R/I$ is strongly F-regular (Theorem \ref{strongly thm}). As a corollary, we prove that the symbolic and ordinary powers coincide in the strongly F-regular case if the height of $I$ is $2$.

\begin{theorem}[Theorem \ref{strongly thm}]
Let $R$ be a regular ring of characteristic $p>0$ (respectively essentially of finite type over a field of characteristic $0$), and let $I$ be an  ideal of height $h\geqslant 2$ such that $R/I$ is strongly F-regular (respectively, has log-terminal singularities). Then $I^{(d)} \subseteq I I^{(d-h+1)}$ for all $d \geqslant h$. In particular, $I^{(n(h-1)+1)} \subseteq I^{n+1}$ for all $n \geqslant 1$.
\end{theorem}

\section{Preliminaries}
Let $A$ be a Noetherian ring of characteristic $p>0$, and denote by $F^e$ the $e$-th iteration \emph{Frobenius} map $F(r)=r^p$. The ring $A$ is called \emph{F-finite} if $A$ is a finite module over itself via the action of the Frobenius map. We say that $A$ is F-pure if for any $A$-module $M$, $F \otimes 1 \!: A \otimes M \longrightarrow A \otimes M$ is injective. F-pure rings were introduced by Hochster and Roberts in \cite{HRFpurity}. Regular rings are F-pure, and squarefree monomial ideals define F-pure rings. For a complete characterization of one-dimensional F-pure rings, see \cite{GotoWatanabe1dimFpure}. 

Given an ideal $I$ in a Noetherian ring of characteristic $p>0$, the Frobenius powers $I^{[p^e]}$ of $I$ are obtained by repeatedly applying the Frobenius map to $I$. If $I= \left(a_1, \ldots, a_n \right)$, then $I^{[p^e]} = \left( a^{p^e}: a \in I \right) = \left( a_1^{p^e}, \ldots, a_n^{p^e} \right)$. The following theorem characterizes ideals that define F-pure rings:

\begin{theorem}[Fedder's Criterion for F-purity, Theorem 1.12 in \cite{fedder}]\label{fedder}
Let $(R, \mm)$ be a regular local ring of characteristic $p>0$. Given an ideal $I$ in $R$, $R/I$ is F-pure if and only if for all $q = p^e \gg 0$,
$$\left(I^{[q]} : I \right) \nsubseteq \mm^{[q]}.$$
\end{theorem}

If $A$ is F-finite and reduced, the ring of $p^e$-roots of $A$ is denoted by $A^{1/p^e}$, and the inclusion $A \hookrightarrow A^{1/p^e}$ can be identified with $F^e$. An F-finite reduced ring $A$ is \emph{strongly F-regular} if given any $f \in A$, $f$ a nonzero divisor in $A$, there exists $q=p^e$ such that the inclusion $f^{1/q} A \longrightarrow A^{1/q}$ splits. Strongly F-regular rings, first introduced in \cite{HHStrong}, are F-pure, normal, and Cohen-Macaulay. Veronese subrings of polynomial rings, locally acyclic cluster algebras \cite{ClusterAlgSingularities} and certain ladder determinantal varieties \cite{GS} are strongly F-regular. See also \cite{SchubertVarieties}.

There is a criterion similar to Theorem \ref{fedder} for strongly F-regular rings: 

\begin{theorem}[Glassbrenner's Criterion for strong F-regularity, \cite{Glassbrenner}]\label{strongly criterion}
Let $(R, \mm)$ be an F-finite regular local ring of prime characteristic $p$. Given a proper radical ideal $I$ of $R$, $R/I$ is strongly F-regular if and only if for each element $c\in R$ not in any minimal prime of $I$, $c \left(I^{[q]} : I \right) \nsubseteq \mm^{[q]}$ for all $q=p^e \gg 0$.
\end{theorem}

\vskip 10mm

Now let $R$ be a Noetherian ring of any characteristic, and $I$ an ideal in $R$. The \emph{$n$-th symbolic power} of $I$ is given by
$$I^{(n)} \quad = \bigcap_{P \in \Ass(R/I)} I^n R_P \cap R.$$
Note that $I^{(n+1)} \subseteq I^{(n)}$ for each $n$, and that given any $a, b \geqslant 1$, $I^{(a)} I^{(b)} \subseteq I^{(a+b)}$. Moreover, $I^n \subseteq I^{(n)}$. The converse containment does not hold in general, so it is natural to ask for what values of $k$ and $m$ is $I^{(k)} \subseteq I^m$. This question is open in general.

\vskip 5mm

To prove our main results, we will need the following results regarding symbolic powers:

\begin{theorem}[Theorem 2.2 (1) in \cite{HHfine}]\label{qout}
Let $R$ be a regular ring of prime characteristic $p>0$, and $I$ an ideal in $R$. Let $h$ be the largest number of generators of $I$ after localizing at an associated prime of $I$. Then for every $n \geqslant 1$ and $q=p^e$, $I^{(hnq)} \subseteq \left( I^{(n)} \right)^{[q]}$.
\end{theorem}

\begin{theorem}[Theorem 2.6 in \cite{comparison}]\label{hn+kn}
Let $I$ be an ideal of a regular ring of prime characteristic $p>0$. Let $h$ be the largest analytic spread of $I_P$ for $P$ an associated prime of $I$. Then for all $n \geqslant 1$ and all $k \geqslant 0$, $I^{(hn+kn)} \subseteq \left( I^{(k+1)} \right)^n$.
\end{theorem}

\begin{lemma}[Lemma 2.4 in \cite{comparison}]\label{inclusion}
Let $R$ be a ring, and $I = \left( u_1, \ldots, u_h \right)$ an ideal in $R$. Then for all integers $t \geqslant 1$ and $k \geqslant 0$,
$$I^{ht+kt-h+1} \subseteq \left( u_1^t, \ldots, u_h^t \right)^{k+1}.$$
Hence, if $R$ has prime characteristic $p>0$ and $q=p^e$, 
$$I^{hq+kq-h+1} \subseteq \left( I^{[q]} \right)^{k+1} = \left( I^{k+1} \right)^{[q]}.$$
\end{lemma}

\begin{lemma}\label{q comparison}
Let $I$ be a radical ideal in a regular ring $R$ of characteristic $p>0$ and $h$ the largest height of a minimal prime of $I$. For all $q=p^e$, 
$$I^{(hq+kq-h+1)} \subseteq \left( I^{(k+1)} \right)^{[q]}.$$
\end{lemma}

\begin{proof}
Taking Frobenius powers preserves the set of associated primes. Indeed, since applying Frobenius is exact, the length of a minimal free resolution of an ideal $J$ coincides with that of its Frobenius powers. Given this, the Auslander-Buchsbaum formula then guarantees that the depth of $R_P/J_P$ coincides with that of $R_P/J_P^{[q]}$ for all $q=p^e$. Moreover, $P$ is an associated prime of an ideal $J$ if and only if $\depth \left( R_P/J_P \right) = 0$. 

Since $I^{(k+1)}$ has no embedded primes, we only need to check the containment at the minimal primes of $I$. So take a minimal prime $P$ of $I$ and replace $R$ by $R_P$ and $I$ by $I_P$. Since $I$ is generated by $h$ or fewer elements, Lemma \ref{inclusion} completes the proof.
\end{proof}

\begin{remark}
	Note that Lemma \ref{q comparison} implies Theorem \ref{qout}, since
	$$hqn \geqslant hq+(n-1)q-h+1$$
	as long as $n, h \geqslant 2$. However, we will use the statement of Theorem \ref{qout} precisely as stated, so we record it here as a means to simplify our proofs.
\end{remark}

\begin{remark}
In all our results, we will assume that $I$ is a radical ideal. In that case, the largest height $h$ of a minimal prime of $I$ coincides with the values of $h$ in Theorem \ref{qout} and Theorem \ref{hn+kn}.
\end{remark}

\section{Ideals defining F-pure rings}

In the next two sections, we prove our main results. We will repeatedly use the following remark:

\begin{remark}
	When $R$ is regular, the Frobenius map is flat, by \cite{Kunz}. As a consequence, $(A:B)^{[q]} = \left( A^{[q]} : B^{[q]} \right)$.
\end{remark}

\begin{lemma}\label{lemma}
Let $R$ be a regular ring of characteristic $p>0$. Let $I$ be a radical ideal in $R$ and $h$ the maximal height of a minimal prime of $I$. For all $n \geqslant 1$ and for all $q = p^e \gg 0$,
$$\left( I^{[q]} : I \right) \subseteq \left( I^{n} : I^{(hn-h+1)} \right)^{[q]}.$$
\end{lemma}

\begin{proof}
Take $s \in \left( I^{[q]} : I \right)$. Then $s I^{(hn-h+1)} \subseteq sI \subseteq I^{[q]}$, so
$$s \left( I^{(hn-h+1)} \right)^{[q]} \subseteq \left( s I^{(hn-h+1)} \right) \left( I^{(hn-h+1)} \right)^{q-1} \subseteq I^{[q]} \left( I^{(hn-h+1)} \right)^{q-1}.$$
We will show that
$$\left( I^{(hn-h+1)} \right)^{q-1} \subseteq \left( I^{n-1} \right)^{[q]},$$
which implies that
$$s \left( I^{(hn-h+1)} \right)^{[q]} \subseteq\left( I^{n} \right)^{[q]},$$
completing the proof.

Notice that

$$\left( I^{(hn-h+1)} \right)^{q-1} \subseteq I^{((hn-h+1)(q-1))}.$$
We claim that $\left( I^{(hn-h+1)} \right)^{q-1} \subseteq I^{((hq-1)(n-1)+h(n-1))}$, and to show that, it is enough to prove that 
$$(hn-h+1)(q-1) \geqslant (hq-1)(n-1)+h(n-1).$$
This holds if $q \geqslant (2h-1)(n-1)+1$.

By Theorem \ref{hn+kn},
$$I^{(h(n-1)+(hq-1)(n-1))} \subseteq \left( I^{(hq)} \right)^{n-1}.$$ 

By Theorem \ref{qout}, $I^{(hq)} \subseteq I^{[q]}$. Then
$$\left( I^{(hn-h+1)} \right)^{q-1} \subseteq \left( I^{[q]} \right)^{n-1}.$$
The result now follows.
\end{proof}

\begin{theorem}\label{Fpure}
Let $R$ be a regular ring of characteristic $p>0$. Let $I$ be an ideal in $R$ with $R/I$ F-pure, and let $h$ be the maximal height of a minimal prime of $I$. Then for all $n \geqslant 1$, $I^{(hn-h+1)} \subseteq I^n$.
\end{theorem}

\begin{proof}
First, note that we can reduce to the local case: in general, to prove an inclusion of ideals $K \subseteq L$ it suffices to prove the containment holds after localizing at an arbitrary associated prime of $L$; moreover, all localizations of an F-pure ring are F-pure \cite[6.2]{HoRo}. So suppose that $(R, \mm)$ is a regular local ring, and that $R/I$ is F-pure.

Fix $n \geqslant 1$, and consider $q$ as in Lemma \ref{lemma}. Then 
$$\left( I^{[q]} : I \right) \subseteq \left( I^{n} : I^{(hn-h+1)} \right)^{[q]}.$$
If $I^{(hn-h+1)} \nsubseteq I^n$, then $\left( I^{n} : I^{(hn-h+1)} \right)^{[q]} \subseteq \mm^{[q]}$, contradicting Fedder's Criterion (Theorem \ref{fedder}).
\end{proof}

\begin{example}[Squarefree monomial ideals]
Let $R$ be a polynomial ring over a field of characteristic $p>0$, and $I$ an ideal generated by squarefree monomials; then $R/I$ is F-pure \cite[Example 2.2]{AnuragUliPure}. Theorem \ref{Fpure} recovers the result that monomial ideals verify Conjecture \ref{Harbourne} in the squarefree case \cite[Example 8.4.5]{Seshadri}. For more on symbolic powers of monomial ideals, see \cite{SymbMon}.
\end{example}

The following example shows that Theorem \ref{Fpure} is sharp for both the case in which $R/I$ is F-pure and where $I$ is a monomial ideal. This example is a special case of the star configurations of points in \cite{HaHu}.

\begin{example}
Let $I \subseteq R := k[x_1, \ldots, x_v]$ be the following ideal:
$$I = \bigcap_{i \neq j} \left( x_i, x_j \right) = \left( x_1 \ldots \hat{x_i} \ldots x_v: 1 \leqslant i \leqslant v \right).$$
Note that the symbolic powers of $I$ can be written as
$$I^{(m)} = \bigcap_{i \neq j} \left( x_i, x_j \right)^m.$$
A monomial $x_1^{a_1} \cdots x_v^{a_v}$ is in $\left( x_i,x_j \right)^m$ if and only if $a_i + a_j \geqslant m$; thus $x_1^{a_1} \cdots x_v^{a_v} \in I^{(m)}$ if and only if $a_i + a_j \geqslant m$ for all $i \neq j$. In particular, $x_1^{n-1} \cdots x_v^{n-1} \in I^{(2n-2)}$. On the other hand, elements of $I$ have degree at least $v-1$, so that elements of $I^n$ have degree at least $(v-1)n=vn-n$. If $n<v$, then $vn-v<vn-n$, and since the degree of $x_1^{n-1} \cdots x_v^{n-1}$ is $v(n-1)=vn-v$, $x_1^{n-1} \cdots x_v^{n-1} \notin I^{n}$. 

Notice that all the associated primes of $I$ have height $2$, so $h=2$. Theorem \ref{Fpure} says that $I^{(2n-1)} \subseteq I^n$ for all $n$, but the previous argument shows that $I^{(2n-2)} \nsubseteq I^n$ for all $n < v$.

This example can be generalized to any $h \geqslant 2$, by taking  $I =\bigcap_{i_1 < \cdots < i_h} \left( x_{i_1}, \ldots, x_{i_h} \right)$. In that case, $I^{(hn-h)} \nsubseteq I^n$ for all $n(h-1)<v$.
\end{example}

\section{Ideals defining strongly F-regular rings}

The main result of this section is the following:

\begin{theorem}\label{strongly thm}
Let $R$ be an F-finite regular ring of characteristic $p>0$, and $I$ an ideal of height $h \geqslant 2$ such that $R/I$ is strongly F-regular. Then $I^{(d)} \subseteq I I^{(d+1-h)}$ for all $d \geqslant h-1$. In particular, 
$$I^{((h-1)n+1)} \subseteq I^{n+1}$$ 
for all $n \geqslant 1$.
\end{theorem}

First, we prove a lemma:

\begin{lemma}\label{strongly lemma}
Let $R$ be a regular ring of characteristic $p>0$, $I$ an ideal in $R$, and $h \geqslant 2$ the maximal height of a minimal prime of $I$. Then for all $d \geqslant h-1$ and for all $q = p^e$,
$$\left( I^{d} : I^{(d)} \right) \left( I^{[q]} : I \right) \subseteq \left(I I^{(d+1-h)} : I^{(d)} \right)^{[q]}.$$
\end{lemma}

\begin{proof}
Let $t \in \left( I^{d} : I^{(d)} \right)$ and $s \in \left( I^{[q]} : I \right)$. We want to show that
$$st \left(  I^{(d)} \right)^{[q]} \subseteq \left( I I^{(d+1-h)} \right)^{[q]}.$$
First, note that
$$st \left(  I^{(d)} \right)^{[q]} \subseteq s \left( t I^{(d)} \right) \left( I^{(d)} \right)^{q-1} \subseteq s I^{d} \left( I^{(d)} \right)^{q-1} \subseteq I^{[q]} I^{d-1} \left( I^{(d)} \right)^{q-1}.$$
Now since $d-1+d(q-1) = dq-1$, we have $I^{d-1} \left( I^{(d)} \right)^{q-1} \subseteq I^{(qd-1)}$. By Lemma \ref{q comparison}, we get 
$$I^{(qd-h+1)} \subseteq \left( I^{(d+1-h)} \right)^{[q]}.$$
As long as $h \geqslant 2$, we have $qd-1 \geqslant qd-h+1$, which implies that $I^{(qd-1)} \subseteq \left( I^{(d+1-h)} \right)^{[q]}$. Then
$$st \left(  I^{(d)} \right)^{[q]} \subseteq I^{[q]} I^{d-1} \left( I^{(d)} \right)^{q-1} \subseteq \left( I I^{(d+1-h)} \right)^{[q]},$$
as desired.
\end{proof}

We begin the proof of Theorem \ref{strongly thm}. We first note that the second statement follows from the first, by induction. To see this, assume we have shown that $I^{(d)} \subseteq I I^{(d+1-h)}$ for all $d \geqslant h-1$. When $d=h$, this means that $I^{(h)} \subseteq I I^{(1)} \subseteq I^2$, which is the statement we are trying to show for the case $n=1$. The induction step follows from choosing $d=(h-1)(n+1)+1$.
  
We prove the first statement by contradiction. As before, we can reduce to the case where $(R, \mm)$ is a regular local ring. Note that strong F-regularity is a local property, that is, $R$ is strongly F-regular if and only if all of its localizations are strongly F-regular \cite[3.1 (a)]{HHStrong}. Suppose that $\left(I I^{(d+1-h)} : I^{(d)} \right) \subseteq \mm$. We can always find an element $t \in \left( I^{d} : I^{(d)} \right)$ not in any minimal prime of $I$. By Lemma \ref{strongly lemma}, 
$$t \left( I^{[q]} : I \right) \subseteq \left(I I^{(d+1-h)} : I^{(d)} \right)^{[q]} \subseteq \mm^{[q]}.$$
By Theorem \ref{strongly criterion}, this contradicts the fact that $R/I$ is strongly F-regular.

\begin{remark}
	Note that if $R$ is a local ring and $R/I$ is strongly F-regular, then $I$ is a prime ideal, so that the value $h$ in Lemma \ref{q comparison} is the height of $I$.
\end{remark}

\begin{corollary}\label{containment cor codim 2}
Let $R$ be a regular ring of characteristic $p>0$, and $I$ a height $2$ prime such that $R/I$ is strongly F-regular. Then $I^{(n)} = I^n$ for all $n \geqslant 2$.
\end{corollary}

The following example suggests that Theorem \ref{strongly thm} might not necessarily be sharp:

\begin{example}[Determinantal ideals]\label{determinantal}
Let $K$ be a field of characteristic $0$ or $p>\min{\lbrace t, n-t \rbrace}$, $X$ a generic $n \times n$ matrix, and let $I_t$ denote the ideal of $t$-minors of $X$ in $R=K[X]$. These ideals of minors define strongly F-regular rings in characteristic $p$, by \cite[7.14]{HoHu3}.

For  which values of $k$ and $m$ do we have $I^{(k)}_t \subseteq I^m_{t}$? We claim that this holds as long as $\frac{n}{t(n-t+1)} k \geqslant m$. 

Fix $k$ and consider $m \leqslant \frac{nk}{t(n-t+1)}$. By \cite[10.4]{determinantalrings}, given $s_i$-minors $\delta_i$, we have $\delta_1 \cdots \delta_u \in I^{(k)}_t$ if and only if $\displaystyle\sum_{i=1}^{u} \max \left\lbrace 0, s_i-t+1 \right\rbrace \geqslant k$, and moreover, $I^{(k)}_t$ is generated by such products. Notice that any factors corresponding to minors of size less than $t$ are irrelevant to determine whether or not the given product is in $I^{(k)}_t$.

So consider $\delta_1, \cdots, \delta_u$, with $\delta_i$ an $s_i$-minor for each $i$, such that $\delta_1 \cdots \delta_u \in I^{(k)}_t$, and write $s=s_1+\cdots+s_u$. Using the remark above, we may assume that $s_i \geqslant t$ for each $i$. We want to show that for all such possible choices of $\delta_i$, $\delta_1 \cdots \delta_u \in I^{m}_t$. Since $\delta_1 \cdots \delta_u \in I^{(k)}_t$, then 
$$\displaystyle\sum_{i=1}^{u} \max \left\lbrace 0, s_i-t+1 \right\rbrace \geqslant k$$
which, since we are assuming $s_i \geqslant t$ for all $i$, can be rewritten as
$$(*) \qquad s \geqslant k+u(t-1).$$
Moreover, $s_i$ is the size of a minor of an $n \times n$ matrix, and thus each $s_i \leqslant n$. In particular, this implies that $un \geqslant s$. Therefore, we must have
$$un \geqslant k+u(t-1),$$
so that
$$u \geqslant \frac{k}{n-t+1}.$$
Combining this with $(*)$, we obtain
$$s \geqslant k+\frac{k}{n-t+1}(t-1) = \frac{nk}{n-t+1}.$$

By \cite[7.3]{DEPdeterminantal}, $I_t^{m} = I_t^{(m)} \cap I_{t-1}^{(2m)} \cap \cdots \cap I_1^{(tm)}$. Thus it suffices to show that $\delta_1 \cdots \delta_u \in I_t^{m}$. In particular, for each $1 \leqslant j \leqslant t$, we need to show that $\delta_1 \cdots \delta_u \in I_j^{{((t-j+1)m)}}$, which is equivalent to the following inequalities:
$$\begin{array}{l}
s \geqslant tm \\ 
s\geqslant (t-1)m + u = tm + (u-m)\\
s \geqslant (t-2)m + 2u = tm + 2(u-m)\\
\qquad \vdots \\ 
s \geqslant m + (t-1)u = tm + (t-1)(u-m).
\end{array}$$
Note that all these inequalities are convex combinations of the first and the last one, so that they are verified as long as
$$\begin{array}{l}
s \geqslant tm \\ 
s \geqslant m + (t-1)u.
\end{array}$$

Since we are assuming that $s \geqslant k + (t-1)u$, then $s \geqslant m + (t-1)u$ as long as $k \geqslant m$, which must be satisfied by any $m$ with $I_t^{(k)} \subseteq I_t^m$. It remains to check that $s \geqslant tm$.

We have seen above that $s \geqslant \frac{nk}{n-t+1}$. Given this, $s \geqslant tm$ is satisfied as long as $\frac{nk}{n-t+1} \geqslant tm$, which can be rewritten as $k \geqslant \frac{(n-t+1)t}{n} \, m$, which was our claim.

We conclude that $I^{(k)} \subseteq I^m$ as long as $k \geqslant \frac{t(n-t+1)}{n} \, m$. This bound is much better than that of Theorem \ref{strongly thm} when $n$ is large and $t$ is close to $\frac{n}{2}$. We also note that the computations above can be used to show that, given $k$, this is actually the best possible value of $m$.

Note that we can obtain the same containment results for the ideal of $t \times t$ minors of a symmetric $n \times n$ matrix or the ideal of $2t$-Pfaffians of a generic $n \times n$ matrix, using Proposition 4.3 and Theorem 4.4 in \cite{MultClassicalVar} for the symmetric case and Theorem 2.1 and Theorem 2.4 in \cite{pfaffians} for the Pfaffians.
\end{example}

\begin{example}[Singh]\label{singh}
Let $R=k[a,b,c,d]$, where $k$ is a field of prime characteristic $p>2$, $n$ be an integer, and let $I$ be the ideal of $2 \times 2$ minors of the $2 \times 3$ matrix
$$\left( \begin{array}{ccc} a^2 & b & d \\ c & a^2 & b^n - d \end{array} \right) .$$
By \cite[Proposition 4.3]{AnuragNotDeform}, $R/I$ is strongly F-regular. Since $I$ is an ideal of height $2$, Corollary \ref{containment cor codim 2} says that $I^{(k)} = I^k$ for all $k$. 

An interesting point connected with this example is the following: the ideal $I$ of maximal minors of a generic $2$ by $3$ matrix defines
a strongly F-regular quotient ring. As we specialize the entries by moding out general choices of elements
in some suitable ideal, the quotient ring will remain strongly F-regular for a while, but not forever. Likewise,
the symbolic powers of the same ideal $I$ are just the regular powers, and this will also remain true
as we specialize, again for a while. Is the point at which the strongly F-regular property breaks the
same as that where the equality of symbolic and ordinary powers no longer remain the same? Is there
something interesting to say on this point?

\end{example}

There are other circumstances under which improved bounds hold. The next theorem
gives one such example:

\smallskip

\begin{theorem}
	Let $R$ be a regular ring and $I$ a radical ideal such that $I^{(k)} = I^k$. Let $h$ be the maximal height of a minimal prime of $I$. Then $I^{(a)} \subseteq I^b$ whenever
	$$a  \geqslant \left \lceil{\frac{b}{k}}\right \rceil (h+k-1).$$
	\end{theorem}

\begin{proof}
Write $b=kn-i$, where $i=0, \ldots, k-1$. Notice that $n= \left \lceil{\frac{b}{k}}\right \rceil$. By Theorem \ref{hn+kn},
	$$I^{((h+k-1)n)} \subseteq \left( I^{(k)} \right)^n = I^{kn} \subseteq I^{b}.$$
\end{proof}

Given two ideals $I$ and $J$ in $R$, $I$ and $J$ are \emph{linked} if there exists a regular sequence $\underline{x} = x_1, \ldots, x_n$ in $I\cap J$ such that $J=(\underline{x}):I$ and $I=(\underline{x}):J$. The \emph{linkage class} of $J$ is the set of ideals $I$ such that there exist ideals $K_0, \ldots, K_s$ with $K_0=J$ and $K_s=I$ and $K_i$ is liked to $K_{i-1}$. An ideal $I$ is \emph{licci} if it is in the linkage class of a complete intersection. For more on linkage, see \cite{HULinkage}. One can also study the behavior of
various types of F-singularities under linkage. See  \cite{LinkageFsingularities}.

\begin{corollary}\label{licci}
	Let $R$ be a regular ring and $I$ a licci radical ideal of height $h$ such that $R/I$ is Gorenstein. Let $h$ be the maximal height of a minimal prime of $I$. If $a \geqslant \left \lceil{\frac{b}{2}}\right \rceil (h+1)$, then $I^{(a)} \subseteq I^b$.
	\end{corollary}

\begin{proof}
By \cite[Theorem 2.8 b) i)]{HU}, $I^{(2)}=I^2$.  The result follows.

\end{proof}

\smallskip

\begin{remark}
We can apply Corollary \ref{licci} to the case of height three Gorenstein ideals, since these are licci \cite{WatanabeJCodim3}. In
the case we take the ideal $I$ generated by generic maximal Pfaffians of a skew symmetric $(2n+1) \times (2n+1)$ matrix, the quotient is a height three strongly F-regular Gorenstein ideal (see Example \ref{determinantal}).  Hence we can compare the two bounds given by Corollary \ref{licci} and Theorem \ref{strongly thm}.
The bound given in Theorem \ref{strongly thm} is that for all $n\geqslant 1$, $I^{(2n-1)}\subseteq I^n$. The bound
of Corollary \ref{licci} is that for all $n \geqslant 1$, $I^{(4\cdot\left \lceil{\frac{n}{2}}\right \rceil)} \subseteq I^n$,
which is slightly worse. Of course as the height increases, the bound given in Corollary \ref{licci} becomes better.
\end{remark}

Other examples where symbolic and ordinary powers agree can be found in \cite{SymbPowersCodim2}.

As a final remark, we note that one may often make a flat extension to study these containments.

\begin{remark}
If $R \longrightarrow S$ is faithfully flat, $P$ is a prime ideal in $R$, and $Q = PS$, then $Q^{(b)} \subseteq Q^a$ implies that $P^{(b)} \subset P^a$. Indeed, $Q \cap R = P$; moreover, nonzero divisors in $R/P$ remain nonzero divisors in $S/Q$ by flatness, so $P^{(b)} \subseteq Q^{(b)} \cap R$. Note also that the value of $h$, the maximal height of an associated prime of $P$, does not increase when passing to $Q$ (cf. \cite[Discussion 2.3 (c)]{comparison}). 
\end{remark}

\section{Equicharacteristic 0}

By using reduction to characteristic $p>0$ techniques, such as in \cite{HHCharZero}, we can generalize the main results to the case of equal characteristic $0$. For information concerning the basic definitions
and relationships see \cite{HW}, or the survey \cite[Appendix C]{SurveyTestIdeals}.

\begin{definition}[dense F-type]
	Let $R$ be a ring essentially of finite type over a field $k$ of characteristic zero. Suppose we are given a model $R_A$ of $R$ over a finitely generated $\mathbb{Z}$-subalgebra $A$ of $k$. We say that $R$ is of \emph{dense F-pure type} (respectively \emph{dense strongly F-regular type}) if there exists a dense subset of closed points $S \subseteq \Spec A$ such that $R_\mu$ is F-pure (respectively strongly F-regular) for all $\mu \in S$.
\end{definition}

\begin{remark}
	Note that the definition above is independent of the choice of $R_A$.
\end{remark}

\begin{theorem}\label{Fpure char 0}
Let $R$ be a regular ring, essentially of finite type over a field of characteristic $0$. Let $I$ be an ideal in $R$ such that $R/I$ is of dense F-pure type, and let $h$ be the largest height of any minimal prime of $I$. Then, for all integers $n \geqslant 1$,
$$I^{(hn-h+1)}\subseteq I^n.$$
\end{theorem}

\begin{proof}
Using the standard descent theory of \cite[Chapter 2]{HHCharZero}, we can reduce the problem to the positive characteristic case, following the same steps as in \cite[Theorem 4.2]{comparison}. After reducing to characteristic $p>0$, the statement follows from Theorem \ref{Fpure}.
\end{proof}

The following conjecture is called the Weak Ordinarity Conjecture. See \cite{Patakfalvi} for a nice
survey on this topic and related matters.  Takagi proved that if the  Weak Ordinarity Conjecture
holds, then log canonical singularities have dense F-pure type:

\begin{conjecture}[Weak ordinarity conjecture] \label{Weak ordinarity conjecture}
Let Y be a smooth, connected projective variety over an algebraically closed field $k$ of characteristic $0$. Given a model $Y_A$ of $Y$ over a finitely generated $\mathbb{Z}$-algebra $A$, the set 
$$\left\lbrace s \in \Spec A \, | \, s \textrm{ is a closed point, the action of Frobenius on } \Ho^{\dim Y_s} \left( Y_s, \mathcal{O}_{Y_s} \right) \textrm{ is bijective} \right\rbrace$$
is dense in $\Spec A$.
\end{conjecture}

\begin{theorem}[Theorem 2.11 in \cite{TakAdjoint}]
	Let X be a log canonical singularity over an algebraically
closed field $k$ of characteristic 0. Given a model $X_A$ of $X$ over a finitely generated $\mathbb{Z}$-algebra $A$,
$$\left\lbrace s \in \Spec A \, | \, s \textrm{ is a closed point such that $X_s$ is F-pure} \right\rbrace$$
is dense in $\Spec A$, if we assume Conjecture \ref{Weak ordinarity conjecture}.
\end{theorem}

Similarly to Theorem \ref{Fpure char 0}, we obtain a characteristic $0$ version of Theorem \ref{strongly thm}. 

\begin{theorem}\label{denseF-reg} Let $R$ be a regular ring, essentially of finite type over a field of characteristic $0$. Let $I$ be an ideal in $R$ such that $R/I$ is of dense strongly F-regular type, and let $h$ be the largest height of any minimal prime of $I$. Then, for all integers $n \geqslant 1$,
$$I^{((h-1)n+1)}\subseteq I^{n+1}.$$
\end{theorem}

\begin{proof} Our assumption means that given models $R_A$ and $(R/I)_A$ of $R$ and $R/I$ over a finitely generated $\mathbb{Z}$-subalgebra $A$ of $k$, there exists a dense subset of closed points $S \subseteq \Spec A$ such that $(R/I)_\mu$ is strongly F-regular for all $\mu \in S$. 
	
As in the proof of \ref{Fpure char 0}, we use the standard descent theory of \cite[Chapter 2]{HHCharZero}, reducing the problem to the positive characteristic case, following the same steps as in \cite[Theorem 4.2]{comparison}. After reducing to characteristic $p>0$, the statement follows from Theorem \ref{strongly thm}. \end{proof}

\smallskip

We may apply this result to the case in which $R/I$ has log-terminal singularities.
We refer to \cite[\S 4.1]{TakagiWatanabe} for a definition of log-terminal singularities. Note that we do not need to assume that $R/I$ is $\mathbb{Q}$-Gorenstein in this application, since De Fernex and Hacon \cite{defernexhacon} have given a definition and treatment of log-terminal singularities without this assumption, and these are of dense strongly F-regular type \cite[Lemma 1.6]{TakagiAdj}. We thank the anonymous referee who pointed this out to us.

\smallskip

\begin{theorem}
Let $R$ be a regular ring, essentially of finite type over a field of characteristic $0$. Let $I$ be an ideal in $R$ such that $R/I$ has log-terminal singularities, and let $h$ be the largest height of any minimal prime of $I$. Then, for all integers $n \geqslant 1$,
$$I^{((h-1)n+1)}\subseteq I^{n+1}.$$
\end{theorem}

\begin{proof}
	By \cite[Theorem 5.2]{Hara}, $R$ is of dense strongly F-regular type. We apply Theorem \ref{denseF-reg} to finish the proof.
\end{proof}

\section*{Funding}

The second author was partially supported by NSF grant DMS-1460638, and thanks  them for their support.

\section*{Acknowledgements}

The first author thanks Alessandro De Stefani, Mel Hochster, Jack Jeffries, Luis N\'{u}\~{n}ez-Betancourt, Alexandra Seceleanu, and Karen Smith for conversations and helpful suggestions. The authors would like to thank Anurag Singh for pointing us to example \ref{singh}.

\bibliographystyle{alpha}
\bibliography{References}

\newcommand{\etalchar}[1]{$^{#1}$}
\begin{thebibliography}{MPG{\etalchar{+}}16}

\bibitem[BH10]{BoH}
Cristiano Bocci and Brian Harbourne.
\newblock Comparing powers and symbolic powers of ideals.
\newblock {\em J. Algebraic Geom.}, 19(3):399--417, 2010.

\bibitem[BMRS15]{ClusterAlgSingularities}
Angélica Benito, Greg Muller, Jenna Rajchgot, and {Karen E.} Smith.
\newblock Singularities of locally acyclic cluster algebras.
\newblock {\em Algebra and Number Theory}, 9(4):913--936, 2015.

\bibitem[BRH{\etalchar{+}}09]{Seshadri}
Thomas Bauer, S~Di Rocco, Brian Harbourne, Micha{\l} Kapustka, Andreas Knutsen,
  Wioletta Syzdek, and Tomasz Szemberg.
\newblock A primer on {S}eshadri constants.
\newblock {\em Contemporary Mathematics}, 496:33, 2009.

\bibitem[BT06]{SchubertVarieties}
Michel Brion and Jesper~Funch Thomsen.
\newblock F-regularity of large {S}chubert varieties.
\newblock {\em American Journal of Mathematics}, 128(4):949--962, 2006.

\bibitem[BV88]{determinantalrings}
Winfried Bruns and Udo Vetter.
\newblock {\em Determinantal rings}.
\newblock 1988.

\bibitem[CEHH16]{SymbMon}
Susan~M Cooper, Robert~JD Embree, Huy~T{\`a}i H{\`a}, and Andrew~H Hoefel.
\newblock Symbolic powers of monomial ideals.
\newblock {\em Proceedings of the Edinburgh Mathematical Society},
  60(1):39--55, 2016.

\bibitem[CFG{\etalchar{+}}16]{SymbPowersCodim2}
Susan Cooper, Giuliana Fatabbi, Elena Guardo, Anna Lorenzini, Juan Migliore,
  Uwe Nagel, Alexandra Seceleanu, Justyna Szpond, and Adam~Van Tuyl.
\newblock Symbolic powers of codimension two {C}ohen-macaulay ideals, 2016.

\bibitem[DEP80]{DEPdeterminantal}
Corrado DeConcini, David Eisenbud, and Claudio Procesi.
\newblock Young diagrams and determinantal varieties.
\newblock {\em Inventiones mathematicae}, 56(2):129--165, 1980.

\bibitem[dFH09]{defernexhacon}
Tommaso de~Fernex and Christopher~D. Hacon.
\newblock Singularities on normal varieties.
\newblock {\em Compos. Math.}, 145(2):393--414, 2009.

\bibitem[DN]{pfaffians}
Emanuela De~Negri.
\newblock K-algebras generated by pfaffians.
\newblock {\em Math. J. Toyama Univ}, 9(l996):105--114.

\bibitem[DSTG13]{counterexamples}
Marcin Dumnicki, Tomasz Szemberg, and Halszka Tutaj-Gasi{\'n}ska.
\newblock Counterexamples to the {$I^{(3)}\subseteq I^2$} containment.
\newblock {\em Journal of Algebra}, 393:24--29, 2013.

\bibitem[Dum15]{Dumnicki2015}
Marcin Dumnicki.
\newblock Containments of symbolic powers of ideals of generic points in {$\Bbb
  P^3$}.
\newblock {\em Proc. Amer. Math. Soc.}, 143(2):513--530, 2015.

\bibitem[ELS01]{ELS}
Lawrence Ein, Robert Lazarsfeld, and Karen~E. Smith.
\newblock {Uniform bounds and symbolic powers on smooth varieties}.
\newblock {\em Inventiones Math}, 144 (2):241--25, 2001.

\bibitem[Fed83]{fedder}
Richard Fedder.
\newblock {{$F$}-purity and rational singularity}.
\newblock {\em Trans. Amer. Math. Soc.}, 278(2):461--480, 1983.

\bibitem[GiW77]{GotoWatanabe1dimFpure}
Shiro Goto and Kei ichi Watanabe.
\newblock The structure of one-dimensional {F}-pure rings.
\newblock {\em Journal of Algebra}, 49(2):415 -- 421, 1977.

\bibitem[Gla96]{Glassbrenner}
Donna Glassbrenner.
\newblock Strongly {F}-regularity in images of regular rings.
\newblock {\em Proceedings of the American Mathematical Society}, 124(2):345 --
  353, 1996.

\bibitem[GS95]{GS}
Donna Glassbrenner and Karen~E. Smith.
\newblock {Singularities of certain ladder determinantal varieties}.
\newblock {\em J. Pure Appl. Algebra}, 101(1):59--75, 1995.

\bibitem[Har98]{Hara}
Nobuo Hara.
\newblock A characterization of rational singularities in terms of injectivity
  of frobenius maps.
\newblock {\em American Journal of Mathematics}, 120(5):981--996, 1998.

\bibitem[HH89]{HHStrong}
Melvin Hochster and Craig Huneke.
\newblock {Tight closure and strong {$F$}-regularity}.
\newblock {\em M{\'e}m. Soc. Math. France (N.S.)}, (38):119--133, 1989.
\newblock Colloque en l'honneur de Pierre Samuel (Orsay, 1987).

\bibitem[HH94]{HoHu3}
Melvin Hochster and Craig Huneke.
\newblock {Tight closure of parameter ideals and splitting in module-finite
  extensions}.
\newblock {\em J. Algebraic Geom.}, 3(4):599--670, 1994.

\bibitem[HH99]{HHCharZero}
Melvin Hochster and Craig Huneke.
\newblock {Tight closure in equal characteristic zero}.
\newblock 1999.

\bibitem[HH02]{comparison}
Melvin Hochster and Craig Huneke.
\newblock {Comparison of symbolic and ordinary powers of ideals}.
\newblock {\em Invent. Math. 147 (2002), no. 2, 349--369}, November 2002.

\bibitem[HH07]{HHfine}
Melvin Hochster and Craig Huneke.
\newblock A fine behavior of symbolic powers.
\newblock {\em Illinois J. Math.}, 51(1):171 -- 183, 2007.

\bibitem[HH13]{HaHu}
Brian Harbourne and Craig Huneke.
\newblock Are symbolic powers highly evolved?
\newblock {\em J. Ramanujan Math. Soc.}, 28A:247--266, 2013.

\bibitem[HR74]{HoRo}
Melvin Hochster and Joel~L. Roberts.
\newblock {Rings of invariants of reductive groups acting on regular rings are
  {C}ohen-{M}acaulay}.
\newblock {\em Advances in Math.}, 13:115--175, 1974.

\bibitem[HR76]{HRFpurity}
Melvin Hochster and Joel~L. Roberts.
\newblock {The purity of the {F}robenius and local cohomology}.
\newblock {\em Advances in Math.}, 21(2):117--172, 1976.

\bibitem[HS15]{HaSe}
Brian Harbourne and Alexandra Seceleanu.
\newblock Containment counterexamples for ideals of various configurations of
  points in {$\bold{P}^N$}.
\newblock {\em J. Pure Appl. Algebra}, 219(4):1062--1072, 2015.

\bibitem[HU87]{HULinkage}
Craig Huneke and Bernd Ulrich.
\newblock The structure of linkage.
\newblock {\em Annals of Mathematics}, 126(2):277--334, 1987.

\bibitem[HU89]{HU}
C.~Huneke and B.~Ulrich.
\newblock Powers of licci ideals.
\newblock In {\em Commutative algebra ({B}erkeley, {CA}, 1987)}, volume~15 of
  {\em Math. Sci. Res. Inst. Publ.}, pages 339--346. Springer, New York, 1989.

\bibitem[HW02]{HW}
Nobuo Hara and Kei-Ichi Watanabe.
\newblock F-regular and {F}-pure rings vs. log terminal and log canonical
  singularities.
\newblock {\em J. Algebraic Geom.}, 11(2):363--392, 2002.

\bibitem[JMV15]{MultClassicalVar}
Jack Jeffries, Jonathan Monta{\~{n}}o, and Matteo Varbaro.
\newblock Multiplicities of classical varieties.
\newblock {\em Proceedings of the London Mathematical Society},
  110(4):1033--1055, 2015.

\bibitem[Kun69]{Kunz}
Ernst Kunz.
\newblock {Characterizations of regular local rings for characteristic {$p$}}.
\newblock {\em Amer. J. Math.}, 91:772--784, 1969.

\bibitem[MPG{\etalchar{+}}16]{LinkageFsingularities}
Linquan Ma, Janet Page, Rebecca~R. G., William Taylor, and Wenliang Zhang.
\newblock F-singularities under generic linkage, 2016.

\bibitem[MS17]{MaSchwede}
Linquan Ma and Karl Schwede.
\newblock Perfectoid multiplier/test ideals in regular rings and bounds on
  symbolic powers.
\newblock {\em ar{X}iv preprint}, 2017.

\bibitem[Pat16]{Patakfalvi}
Zsolt Patakfalvi.
\newblock Frobenius techniques in birational geometry, 2016.

\bibitem[Sin99]{AnuragNotDeform}
Anurag~K. Singh.
\newblock {{$F$}-regularity does not deform}.
\newblock {\em Amer. J. Math.}, 121(4):919--929, 1999.

\bibitem[ST12]{SurveyTestIdeals}
Karl Schwede and Kevin Tucker.
\newblock {A survey of test ideals}.
\newblock In {\em {Progress in commutative algebra 2}}, pages 39--99. Walter de
  Gruyter, Berlin, 2012.

\bibitem[SW07]{AnuragUliPure}
Anurag~K. Singh and Uli Walther.
\newblock {Local cohomology and pure morphisms}.
\newblock {\em Illinois J. Math.}, 51(1):287--298 (electronic), 2007.

\bibitem[Tak04]{TakagiAdj}
Shunsuke Takagi.
\newblock {F-singularities of pairs and inversion of adjunction of arbitrary
  codimension}.
\newblock {\em Invent. Math.}, 157(1):123--146, 2004.

\bibitem[Tak13]{TakAdjoint}
Shunsuke Takagi.
\newblock Adjoint ideals and a correspondence between log canonicity and
  {F}-purity.
\newblock {\em Algebra \& Number Theory}, 7:917--942, 2013.

\bibitem[TiW14]{TakagiWatanabe}
Shunsuke Takagi and Kei ichi Watanabe.
\newblock F-singularities: applications of characteristic $p$ methods to
  singularity theory, 2014.

\bibitem[TY07]{takagiyoshida}
Shunsuke Takagi and Ken-ichi Yoshida.
\newblock Generalized test ideals and symbolic powers.
\newblock {\em arXiv preprint math/0701929}, 2007.

\bibitem[Wat73]{WatanabeJCodim3}
Junzo Watanabe.
\newblock A note on {G}orenstein rings of embedding codimension three.
\newblock 50:227--232, 006 1973.

\end{thebibliography}

\end{document}